\documentclass[12pt]{article}
\usepackage[T2A]{fontenc}
\usepackage[english]{babel}
\usepackage[tbtags]{amsmath}
\usepackage{amsfonts,amssymb}
\usepackage{amssymb}
\usepackage{graphicx}
\usepackage{amscd}

\usepackage{fullpage}
\usepackage{float}

\usepackage{graphics,amsmath,amssymb}
\usepackage{amsthm}
\usepackage{amsfonts}
\usepackage{mathrsfs}

\title{Integer properties of a composition of exponential generating functions}

\author{D.~V.~Kruchinin\\
\small Tomsk State University of Control Systems and Radioelectronics, Russian Federation\\
\small \texttt{kruchininDm@gmail.com}\\
}

\begin{document}
\maketitle

\begin{abstract}
In this paper, we study a composition of exponential generating functions.
We obtain new properties of this composition, which allow to distinguish prime numbers  from composite numbers.
Using the result of paper we get the known properties of the Bell numbers(Touchard's Congruence for $k=0$) and  new properties of the Euler numbers.

Key words: exponential generating function, composition of generating functions, composita, primality, Touchard's Congruence, Bell numbers, Euler numbers.
\end{abstract}

\theoremstyle{plain}
\newtheorem{theorem}{Theorem}
\newtheorem{corollary}[theorem]{Corollary}
\newtheorem{lemma}[theorem]{Lemma}
\newtheorem{proposition}[theorem]{Proposition}

\theoremstyle{definition}
\newtheorem{definition}[theorem]{Definition}
\newtheorem{example}[theorem]{Example}
\newtheorem{conjecture}[theorem]{Conjecture}
\theoremstyle{remark}
\newtheorem{remark}[theorem]{Remark}

\newtheorem{Theorem}{Theorem}[section]
\newtheorem{Proposition}[Theorem]{Proposition}
\newtheorem{Corollary}[Theorem]{Corollary}

\theoremstyle{definition}
\newtheorem{Example}[Theorem]{Example}
\newtheorem{Remark}[Theorem]{Remark}
\newtheorem{Problem}[Theorem]{Problem}
\newtheorem{state}[Theorem]{Statement}
\makeatletter
\def\rdots{\mathinner{\mkern1mu\raise\p@
\vbox{\kern7\p@\hbox{.}}\mkern2mu
\raise4\p@\hbox{.}\mkern2mu\raise7\p@\hbox{.}\mkern1mu}}
\makeatother

\section{Introduction}
\label{sec:in}
Generating functions are a powerful tool for solving problems in number theory, combinatorics, algebra, probability theory. One of the advantages of the generating functions is that an infinite number sequence can be represented in the form of single expression. 

There are various types of generating functions: ordinary, exponential, Dirichlet, Poisson, etc. In this paper we consider  exponential generating functions.
\begin{definition}
 		Power series of the form
 			\begin{equation}
 			\label{formula1}
 			\sum_{n=1}^\infty\frac{a(n)}{n!}x^{n}
 			\end{equation} 
 		 is a \textit{exponential generating function}, where $a(n)$ is an integer sequence.
\end{definition}

\section{Preliminary}

V.~V.~Kruchinin \cite{KruCompositae} introduced the notion of the \emph{composita} of a given generating function 
\begin{math}
F(x)=\sum_{n>0}f(n)x^n.
\end{math}

Suppose 
\begin{math}
F(x) = \sum_{n>0} f(n) x^n
\end{math}
 is the generating function, in which there is no free term $f(0)=0$.
From this generating function we can write the following condition:
\begin{displaymath}
[F(x)]^k=\sum_{n>0} F(n,k)x^n.
\end{displaymath}
The expression $F(n,k)$ is the \emph{composita} and it's denoted by $F^{\Delta}(n,k) $. 

Also the \emph{composita} can be written based on the composition of $n$:
\begin{definition}
The composita is the function of two variables defined by
\begin{equation}
\label{Fnk0}F^{\Delta}(n,k)=\sum_{\pi_k \in C_n}{f(\lambda_1)f(\lambda_2) \cdots f(\lambda_k)},
\end{equation}
where $C_n$ is the set of all compositions of an integer $n$, $\pi_k$ is the composition
\begin{math}
\sum_{i=1}^k\lambda_i=n
\end{math}
 into $k$ parts exactly.
\end{definition}

For instance, we obtain the composita of the exponential generating function 
\begin{math}
F(x)=e^x-1.
\end{math}

Raising this generating function to the power of $k$ and applying the binomial theorem, we obtain 
\begin{displaymath}
\left(e^x-1\right)^k=\sum_{j=0}^k {k \choose j}e^{xj}(-1)^{k-j}.
\end{displaymath}
Since 
\begin{displaymath}
\left(e^{x} \right)^k=\sum_{n\geq 0} \frac{k^n}{n!}x^n,
\end{displaymath}
we get
\begin{displaymath}
F^{\Delta}(n,k)=\sum_{j=0}^k {k \choose j}\frac{j^n}{n!}(-1)^{k-j}.
\end{displaymath}
or
since the general formula for the Stirling numbers of the second kind is given as follows:
\begin{displaymath}
\genfrac{\{}{\}}{0pt}{}{n}{k}=\frac{1}{k!}\sum_{j=0}^k(-1)^{k-j}\binom{k}{j}j^n,
\end{displaymath}
we have
\begin{equation}
\label{exp(x)-1}
F^{\Delta}(n,k)=\frac{k!}{n!}\genfrac{\{}{\}}{0pt}{}{n}{k}.
\end{equation}

Here 
\begin{math}
\genfrac{\{}{\}}{0pt}{}{n}{k}
\end{math}
stand for the Stirling numbers  of the second kind(see \cite{Comtet_1974,ConcreteMath}).
The Stirling numbers of the second kind 
\begin{math}
\genfrac{\{}{\}}{0pt}{}{n}{k}=S(n,k)
\end{math}
count the number of ways to partition a set of $n$ elements into $k$ nonempty subsets.

Calculation of the composita  is  essential for obtaining the coefficients function of a composition of generating functions.

Starting with functions $f(n)$, $r(n)$ and their generating functions
\begin{math}
F(x)=\sum_{n\geq1}f(n)x^{n}
\end{math},
\begin{math}
R(x)=\sum_{n\geq0}r(n)x^{n}
\end{math}
accordingly, we consider a composition of the generating functions 
\begin{math}
H(x)=R\left( F(x)\right)
\end{math}.
For the generating function 
\begin{math}
H(x)=\sum_{n\geq0}h(n)x^{n}
\end{math},
the coefficients function  is determined by the expression:
\begin{equation}
	\label{composition}
 	h(n)=\sum^{n}_{k=1}\sum_{\lambda_i>0 \atop 		\lambda_1+\lambda_2+\ldots+\lambda_k=n}f(\lambda_{1})f(\lambda_{2})\ldots f(\lambda_{k})r(k)=\sum^{n}_{k=1}F^{\Delta}(n,k)r(k),\quad h(0)=r(0).
\end{equation}

\section{Main results}
In this section we consider a composition of exponential generating function and its integer properties. 

\begin{theorem}
Suppose
\begin{math}
E(x)=\sum_{n>0} e(n)\frac{x^n}{n!}
\end{math}
is an exponential generating function, and
\begin{math}
E^{\Delta}(n,k)
\end{math}
is the composita of $E(x)$.
Then the expression
\begin{equation}
\label{Enk}
\frac{n!}{k!}E^{\Delta}(n,k) 
\end{equation}
is integer for $k\leq n$
\end{theorem}

\begin{proof} Let us consider a composition of exponential generating functions 
\begin{math}
A(E(x))
\end{math},
 where 
\begin{math}A(x)=\sum_{n\geq 0} a(n)\frac{x^n}{n!}
\end{math},
 $a(n)$ is integer. According to \cite{Stanley_2}, the composition of exponential generating functions is a exponential generating function
\begin{displaymath}
G(x)=A(E(x)),
\end{displaymath} 
where 
\begin{math}
G(x)=\sum_{n\geq 0} g(n)\frac{x^n}{n!}
\end{math}.

Using formula (\ref{composition}), we get
\begin{displaymath}
\frac{g(n)}{n!}=\sum_{k=1}^n E^{\Delta}(n,k)\frac{a(k)}{k!}.
\end{displaymath} 

Therefore, the expression 
\begin{displaymath}
\sum_{k=1}^n E^{\Delta}(n,k)a(k)\frac{n!}{k!}
\end{displaymath} 
is integer. 

Since $g(k)$ can be any integer, the expression (\ref{Enk}) is integer.

The theorem is proved. 
\end{proof}

\begin{corollary}
\label{cor1}
Suppose $E^{\Delta}(n,k)$ is the composita of an exponential generating function.
Then  the expression 
\begin{equation}\label{coEnk}
\sum_{k=2}^{n-1} E^{\Delta}(n,k)\frac{(n-1)!}{k!} 
\end{equation}
is integer for all prime $n$.
\end{corollary}

\begin{proof} 
Let us consider the following cases:
\begin{itemize}
\item For $k=1$. The composita equals to
\begin{math}
E^{\Delta}(n,1)=\frac{e(n)}{n!}.
\end{math}
Hence, the expression 
\begin{displaymath}
 E^{\Delta}(n,k)\frac{(n-1)!}{k!}=\frac{e(n)}{n}
\end{displaymath}
is not integer for $k=1$.

\item For $k=n$. The composita equals to
\begin{math}
E^{\Delta}(n,n)={e(1)^n}.
\end{math}
Hence, the expression 
\begin{displaymath}
 E^{\Delta}(n,k)\frac{(n-1)!}{k!}=\frac{e(1)^n}{n}
\end{displaymath}
is not integer for $k=n$.

\item For $1<k<n$.
According to (\ref{Fnk0}), the composita equals to
\begin{displaymath}
E^{\Delta}(n,k)=\sum_{\pi_k\in C_n} \frac{e(\lambda_1)e(\lambda_2)\ldots e(\lambda_k)}
{\lambda_1! \lambda_2!\ldots \lambda_k!}.
\end{displaymath}
From $\pi_k$ is the composition 
\begin{math}
\lambda_1+\lambda_2+\ldots+\lambda_k=n
\end{math}
and $k>1$ it follows that no exists $\lambda_i$ such that 
\begin{math}
\lambda_i\neq n.
\end{math}
Hence, the expression 
\begin{displaymath}
\sum_{\pi_k\in C_n} \frac{e(\lambda_1)e(\lambda_2)\ldots e(\lambda_k)}
{\lambda_1! \lambda_2!\ldots \lambda_k!}\frac{(n-1)!}{k!}
\end{displaymath}
is integer for all prime $n$ and for $1<k<n$.

Therefore, the expression 
\begin{displaymath}
\sum_{k=2}^{n-1} E^{\Delta}(n,k)\frac{(n-1)!}{k!} 
\end{displaymath}
is integer for all prime $n$.
\end{itemize}
\end{proof}

We represent the expression (\ref{coEnk})
\begin{equation}\label{coEnksum}
\frac{1}{n}\left(n!\sum_{k=1}^{n} E^{\Delta}(n,k)\frac{1}{k!}-e(n)-e(1)^n\right),
\end{equation}
where
\begin{displaymath}
g(n)=n!\sum_{k=1}^{n} E^{\Delta}(n,k)\frac{1}{k!}
\end{displaymath}
is the coefficients function of the composition of generating functions 
\begin{math}
G(x)=\exp(E(x)).
\end{math}

In the general case it also holds.
For the composition of generating functions
\begin{displaymath}
G(x)=A(E(x))=\sum_{n\geq 0} g(n)\frac{x^n}{n!},
\end{displaymath}
where
\begin{displaymath}
A(x)=\sum_{n\geq 0} a(n)\frac{x^n}{n!},
\qquad
E(x)=\sum_{n\geq 1} e(n)\frac{x^n}{n!}
\end{displaymath}
are exponential generating functions, the expression 
\begin{equation}
\frac{1}{n}\left(g(n)-e(n)a(1)-e(1)^na(n)\right)
\end{equation}
is integer for all prime $n$. 

As applications of Corollary \ref{cor1}, we consider the following examples.

\begin{example}
Consider the following composition of generating function:
\begin{displaymath}
G(x)=e^{e^x-1}=\sum_{n\geq 0} g(n)\frac{x^n}{n!}.
\end{displaymath}

The composita of
\begin{math}
E(x)=\exp(x)-1
\end{math}, according to (\ref{exp(x)-1}),
is equal to
\begin{displaymath}
E^{\Delta}(n,k)=\frac{k!}{n!}\genfrac{\{}{\}}{0pt}{}{n}{k}.
\end{displaymath}

Then for the composition
\begin{math}
e^{e^x-1},
\end{math}
the expression
\begin{displaymath}
\frac{1}{n}\left(n!\sum_{k=1}^n \frac{k!}{n!}\genfrac{\{}{\}}{0pt}{}{n}{k}\frac{1}{k!}-1-1^n\right)
\end{displaymath}
is integer for all prime $n$.

Since $n>0$, we get that
the expression
\begin{equation}
\label{Touchard}
(B_n-2)\equiv 0 \mod n.
\end{equation}
is integer for all prime $n$.

Here $B_n$ are the Bell numbers (counting the ways to
partition a set of $n$ elements) \cite{Comtet_1974,ConcreteMath,Bell_1934}.

In 1933 J. Touchard \cite{Tou} proved the next congruence for the Bell numbers:
\begin{equation}
B_{n+k}\equiv B_{k+1}+B_k\pmod{n}
\end{equation}
for any prime number $p$.

The expression (\ref{Touchard}) is a special case of Touchard's Congruence (for $k=0$).
\end{example}

\begin{example}
Let us consider the following composition of exponential generating functions
\begin{displaymath}
E(x)=e^{x+\frac{1}{2}x^2+\frac{1}{6}x^3}.
\end{displaymath}

This generating function generates the sequence of integers (A001333) \cite{oeis}.
\begin{displaymath}
\left[1, 1, 2, 5, 14, 46, 166, 652, 2780, 12644, 61136, 312676, 1680592, 9467680, 55704104, \ldots \right]
\end{displaymath}

According to \cite{KruCompositae}, the composita  of the generating function 
\begin{math}
E(x)=x+\frac{1}{2}x^2+\frac{1}{6}x^3
\end{math}
has the following form
\begin{displaymath}
E^{\Delta}(n,k)=\sum_{j=0}^{k}{{{j}\choose{n-3\,k+2\,j}}\,3^{j-k}\,{{k}\choose{j}}
 \,2^{-n+2\,k-j}}. 
\end{displaymath}

Then, according to (\ref{composition}), for finding the composition we use the following expression
\begin{displaymath}
g(n)=n!\sum_{k=1}^{n}\frac{1}{k!}\sum_{j=0}^{k}{{j}\choose{n-
 3\,k+2\,j}}\,3^{j-k}\,{{k}\choose{j}}\,2^{-n+2\,k-j}.
\end{displaymath}

Hence, the expression
\begin{equation}
(n-1)!\sum_{k=2}^{n-1}\frac{1}{k!}\sum_{j=0}^{k}{{j}\choose{n-
 3\,k+2\,j}}\,3^{j-k}\,{{k}\choose{j}}\,2^{-n+2\,k-j}
\end{equation}
is integer for all prime $n$.

A few initial terms of this expression are shown below (starting with $n=1$):
\begin{displaymath}
\left[0, 0, 1, \frac{13}{4}, 9, \frac{55}{2}, 93, \frac{2779}{8}, \frac{12643}{9}, \frac{12227}{2}, 28425, \frac{560197}{4}, 728283,\ldots \right] 
\end{displaymath}
\end{example}

\begin{example}
Let us consider the following composition of exponential generating functions
\begin{displaymath}
E(x)=e^{artanh(x)}.
\end{displaymath}

This generating function generates the sequence of integers (A000246) \cite{oeis}.
\begin{displaymath}
\left[1, 1, 1, 3, 9, 45, 225, 1575, 11025, 99225, 893025, 9823275, 108056025, 1404728325, \ldots \right]
\end{displaymath}

The composita  of the generating function 
\begin{math}
E(x)=artanh(x)
\end{math}
has the following form
\begin{displaymath}
E^{\Delta}(n,k)=k!\,\sum_{m=k}^{n}\frac{2^{m-k}}{m!}\genfrac{[}{]}{0pt}{}{m}{k}\,
 {{n-1}\choose{m-1}}.
\end{displaymath}

Here
\begin{math}
\genfrac{[}{]}{0pt}{}{m}{k}
\end{math} 
are the Stirling
numbers of the first kind with parameters $m$ and $k$ (counting the ways
to partition a set of $m$ elements into $k$ blocks) \cite{Comtet_1974}

Then, according to (\ref{composition}), for finding the composition we use the following expression
\begin{displaymath}
g(n)=n!\sum_{k=1}^{n}\sum_{m=k}^{n}\frac{2^{m-k}}{m!}\genfrac{[}{]}{0pt}{}{m}{k}\,
 {{n-1}\choose{m-1}}.
\end{displaymath}

Hence, the expression
\begin{equation}
(n-1)!\sum_{k=2}^{n-1}\sum_{m=k}^{n}\frac{2^{m-k}}{m!}\genfrac{[}{]}{0pt}{}{m}{k}\,
 {{n-1}\choose{m-1}}
\end{equation}
is integer for all prime $n$.
\end{example}

Now let us consider the composition of generating functions
\begin{math}
G(x)=B(E(x))=\sum_{n\geq 0} g(n)\frac{x^n}{n!},
\end{math}
where
\begin{math}
E(x)=\sum_{n>0} e(n)\frac{x^n}{n!}
\end{math}
is an exponential generating function,
\begin{math}
B(x)=\sum_{n\geq0} b(n)x^n
\end{math}
is an ordinary generating function with integer coefficients.

\begin{theorem}
\label{Thm2}
For the composition
\begin{math}
G(x)=B(E(x))=\sum_{n\geq 0} g(n)\frac{x^n}{n!},
\end{math}
the expression
\begin{equation}
\label{Enk2}
\frac{1}{n}\left(g(n)-e(n)b(1)\right)
\end{equation}
is integer for all prime $n$.
\end{theorem}

\begin{proof}
From (\ref{composition}) it follows that
\begin{displaymath}
g(n)=n!\sum_{k=1}^nb(k)\sum_{\pi_k\in C_n} \frac{e(\lambda_1)e(\lambda_2)\ldots e(\lambda_k)}
{\lambda_1! \lambda_2!\ldots \lambda_k!}.
\end{displaymath}

Then
\begin{displaymath}
n!\sum_{\pi_k\in C_n} \frac{e(\lambda_1)e(\lambda_2)\ldots e(\lambda_k)}
{\lambda_1! \lambda_2!\ldots \lambda_k!}=\sum_{\pi_k\in C_n}{n \choose {\lambda_1, \lambda_2\ldots \lambda_k} }e(\lambda_1)e(\lambda_2)\ldots e(\lambda_k).
\end{displaymath}

Since $n$ is prime, the multinomial coefficient is divided by $n$ evenly for $k>1$.
Therefore,
the expression
\begin{displaymath}
\frac{1}{n}\left(n!\sum_{k=1}^n E^{\Delta}(n,k)b(k)-e(n)b(1)\right)
\end{displaymath}
is integer for all prime $n$.

The theorem is proved.
\end{proof}

As applications of Theorem \ref{Thm2}, we consider the following example.
\begin{example}
Let us consider the generating function for the Euler numbers A000111 \cite{oeis,Stanley_1}
\begin{equation}
\frac{1}{1-\sin(x)}=\sum_{n\geq 0} E(n+1)\frac{x^n}{n!}
\end{equation}

Since
\begin{displaymath}
E(x)=\sin(x)=\sum_{n\geq0}\frac{\left((-1)^{n-1}+1\right)(-1)^{\frac{n+1}{2}+n}}{2}\frac{x^n}{n!},
\end{displaymath}
the expression 
\begin{equation}
E(n+1)-\frac{\left((-1)^{n-1}+1\right)(-1)^{\frac{3n+1}{2}}}{2} \equiv 0 \mod n 
\end{equation}
holds for all prime $n$.
\end{example}


\begin{thebibliography}{9}

\bibitem{KruCompositae}
V.~V. Kruchinin.
\newblock Compositae and their properties.
\newblock preprint.
\newblock URL {http://arxiv.org/abs/1103.2582}.

\bibitem{Comtet_1974}
L.~Comtet.
\newblock \emph{Advanced Combinatorics}.
\newblock D. Reidel Publishing Company, 1974.

\bibitem{ConcreteMath}
R.~L. Graham, D.~E. Knuth, and O.~Patashnik.
\newblock \emph{Concrete Mathematics}.
\newblock Addison-Wesley, 1989.

\bibitem{Stanley_2}
R.~P. Stanley.
\newblock \emph{Enumerative combinatorics 2}.
\newblock Cambridge Studies in Advanced Mathematics. Cambridge University
  Press, 1999.


\bibitem{Bell_1934}
E.~T. Bell.
\newblock Exponential numbers.
\newblock \emph{Amer. Math. Monthly}, 41:\penalty0 411--419,, 1934.

\bibitem{Tou}
J.~Touchard.
\newblock Propri\'et\'es arithm\'etiques de certains nombres r\'ecurrents.
\newblock \emph{Ann. Soc. Sci. Bruxelles A}, 53:\penalty0 21--31, 1933.

\bibitem{oeis}
J.~A. Sloane.
\newblock The on-line encyclopedia of integer sequences.
\newblock Published electronically at http://oeis.org/, 2012.
\newblock URL {www.oeis.org}.

\bibitem{Stanley_1}
R.~P. Stanley.
\newblock \emph{Enumerative combinatorics 1}.
\newblock Cambridge Studies in Advanced Mathematics. Cambridge University
  Press, 2nd edition, 2011.

\end{thebibliography}
\end{document}